\documentclass[letterpaper, 10pt, conference]{ieeeconf}

\IEEEoverridecommandlockouts
\usepackage{cite}
\usepackage{amsmath,amssymb,amsfonts}
\usepackage{textcomp}
\usepackage{graphicx}
\usepackage{mathtools}
\usepackage[hidelinks]{hyperref}

\newtheorem{theorem}{Theorem}[section]

\newtheorem{proposition}[theorem]{Proposition}
\newtheorem{corollary}[theorem]{Corollary}

\newcommand{\R}{\mathbb{R}}
\newcommand{\T}{\mathbb{T}}
\newcommand{\Pcal}{\mathcal{P}}
\newcommand{\F}{\mathcal{F}}
\newcommand{\Ent}{\operatorname{Ent}}
\newcommand{\potential}[1][\mu]{\dot{H}^1_{#1}}
\newcommand{\rieszmap}[1][\mu]{\mathcal{I}_{#1}}
\newcommand{\dotHk}[1][k]{\dot{H}^{#1}}
\newcommand{\ind}{\mathbf{1}}

\begin{document}

\title{\LARGE \bf Control Strategies for Multi-Species Wasserstein Gradient Flows}

\author{Dante Kalise,
        Lucas M. Moschen,
        and Grigorios A. Pavliotis
\thanks{Manuscript received on September 10th, 2026.
D. Kalise is supported by the Air Force Office of Scientific Research under award number FA8655-26-1-B011.
L. Moschen is funded by the Roth Scholarship at Imperial College London with travel funds from ICL-CNRS.
G. Pavliotis is partially supported by an ERC-EPSRC Frontier Research Guarantee through Grant No. EP/X038645, ERC Advanced Grant No. 247031 and a Leverhulme Trust Senior Research Fellowship, SRF\textbackslash{}R1\textbackslash{}241055.}%
\thanks{D. Kalise, L. M. Moschen, and G. A. Pavliotis are with the Department of Mathematics, Imperial College London, London, UK (e-mail: \{d.kalise-balza, lmm122, pav\}@ic.ac.uk).}%
}

\maketitle
\thispagestyle{empty}

\begin{abstract}
    We stabilize stationary states of coupled multi-species McKean--Vlasov equations by feedback control, extending a spectral approach based on the Wasserstein Hessian and a Riccati equation from a single density to a product of Wasserstein spaces.
    We give necessary and sufficient conditions for local exponential stabilization at a prescribed rate when only one species is actuated, together with the minimum number of control inputs required.
    For two species, the condition reduces to a rank condition on the Fourier coefficients of the cross-interaction kernel over the eigenspaces of the unactuated species' linearized dynamics.
    A two-community noisy Kuramoto model and a cell-sorting model illustrate the results.
\end{abstract}

\section{Introduction}
\label{sec:introduction}

Coupled multi-species McKean--Vlasov systems describe populations whose densities evolve through diffusion and interactions within and across species, and arise in models of neuronal synchronization~\cite{rohling2020two, Meylahn2020} and the spatial organization of interacting species~\cite{GiuntaHillenLewisPotts2024, CarrilloSalmaniw2025}.
The well-posedness, stationary states, and long-time behavior of multi-species McKean--Vlasov equations, along with their derivation as mean-field limits, are studied in~\cite{duong2025multi} and the references therein.
Related multi-species Wasserstein dynamics, including examples with cross-diffusion and nonlocal interactions, are considered in~\cite{CongerHoffmanMazumdarRatliff2025}, while aggregation-diffusion systems have been investigated regarding their stability, bifurcation structure, and pattern formation in~\cite{GiuntaHillenLewisPotts2024, CarrilloSalmaniw2025}.

In this paper, we address the control of such mean-field multi-species systems via feedback control and Wasserstein gradient flows. 
Mean-field control and game formulations for systems with several populations are studied in~\cite{BensoussanHuangLauriere2018}, while~\cite{AlbiKalise2018} considers feedback control acting on one population and influencing another through their interaction in a leader-follower setting. 
However, to the best of our knowledge, local exponential stabilization of multi-species Wasserstein gradient flows under single-species actuation has not been characterized. In~\cite{KaliseMoschenPavliotis2026}, we established local exponential stabilization of Wasserstein gradient flows via feedback control for the case of a single density.
In that work, the feedback construction combined the spectral properties of the Wasserstein Hessian of the underlying free energy with an algebraic Riccati equation.
Here, we extend this approach to coupled multi-species McKean--Vlasov equations.

When the interactions satisfy a reciprocity condition, the coupled system is a Wasserstein gradient flow on a product of Wasserstein spaces, and the Hessian-based control design from \cite{KaliseMoschenPavliotis2026} extends to this setting.
Beyond this extension, the multi-species structure raises the question of whether cross-interactions can transmit the effect of a control acting on only one species to the remaining species.
This mechanism of selective actuation arises, for example, in the suppression of pathological neural synchronization, where stimulation may be applied to one interacting neuronal population while the activity of another is monitored~\cite{TukhlinaRosenblum2008}.
In this work, we characterize when such single-species control can stabilize the full system and determine the minimum number of control inputs required for stabilization.
The rest of the paper is organized as follows.
Section~\ref{sec:multispecies} introduces the multi-species free energy and realizes the associated Wasserstein Hessian as a self-adjoint operator on a product space.
Section~\ref{sec:hessian-feedback} develops the feedback design under full actuation and characterizes $\delta$-stabilization through a single species.
Section~\ref{sec:numerics} applies these results in two examples.
Section~\ref{sec:conclusion} concludes.

\paragraph*{Notation}
We write $\T^d \coloneqq (\R/\mathbb{Z})^d$ and $\Pcal(\T^d)$ for the space of Borel probability measures on $\T^d$.
We identify functions differing by constants and set $\dotHk[k](\T^d) \coloneqq H^k(\T^d)/\R$.
For a positive density $w$, $\dotHk[1]_w(\T^d)$ is endowed with the inner product $\langle\phi, \psi\rangle_{\dotHk[1]_w} \coloneqq \int_{\T^d} \nabla\phi \cdot \nabla\psi \, w \, dx$, and we identify each equivalence class with its representative of zero $w$-mean.

\section{Multi-Species Wasserstein Dynamics}
\label{sec:multispecies}

Consider $N \ge 2$ interacting species on $\T^d$.
For each species $i$, let $\sigma_i > 0$ be its diffusion strength, $V_i \in C^2(\T^d)$ its confinement potential, and $W_{ij} \in C^2(\T^d)$ the interaction kernel with $j$.
For $\mu = (\mu_1, \ldots, \mu_N) \in \Pcal(\T^d)^N$, set
\begin{equation*}
    \begin{aligned}
        \F(\mu_1, \ldots, \mu_N) &\coloneqq \sum_{i=1}^N \sigma_i \Ent(\mu_i) + \sum_{i=1}^N \int_{\T^d} V_i \, d\mu_i  \\
        &+ \frac{1}{2} \sum_{i,j=1}^N \iint_{\T^d \times \T^d} W_{ij}(x-y) \, d\mu_i(x) d\mu_j(y).
    \end{aligned}
\end{equation*}
Here, $\Ent(\mu) \coloneqq \int_{\T^d} \mu\log\mu\,dx$ if $\mu \ll dx$ and $\Ent(\mu) \coloneqq +\infty$ otherwise, where an absolutely continuous measure and its density are denoted by the same symbol.
We assume throughout the {\em reciprocity condition} $W_{ji}(z) = W_{ij}(-z)$.
The first variation of $\F$ with respect to $\mu_i$ is
\begin{equation}
    \label{eq:first_variation_ms}
    \frac{\delta\F}{\delta\mu_i}(\mu) = \sigma_i(\log\mu_i + 1) + V_i + \sum_{j=1}^N W_{ij} \ast \mu_j.
\end{equation}
The corresponding Wasserstein gradient flow is
\begin{equation}
    \label{eq:multispecies_pde}
    \partial_t\mu_i = \nabla \cdot \left(\mu_i \nabla\frac{\delta\F}{\delta\mu_i}(\mu)\right), \qquad \mu_i \in \Pcal(\T^d)
\end{equation}
for $i=1, \ldots, N$.
We first record the existence and the regularity of the stationary states of~\eqref{eq:multispecies_pde}, which are needed to define the spaces on which the Wasserstein Hessian acts.

\begin{proposition}
    \label{prop:stationary_states}
    The free energy $\F$ admits a minimizer, which is a stationary state of~\eqref{eq:multispecies_pde}.
        Moreover, a vector of probability densities $\bar\mu$ is stationary if and only if
        \begin{equation}
            \label{eq:kirkwood_monroe}
            \bar\mu_i \propto \exp\left[-\frac{1}{\sigma_i} \left(V_i + \sum_{j=1}^N W_{ij} \ast \bar\mu_j\right)\right],
        \end{equation}
        for each $i = 1, \ldots, N$.
        Every stationary state satisfies $\bar\mu_i \in C^2(\T^d)$ and $c_i \leq \bar\mu_i \leq C_i$ for constants $c_i, C_i > 0$.
\end{proposition}

\begin{proof}
    The space $\Pcal(\T^d)^N$ is compact under any product $\operatorname{W}_1$ metric by~\cite[Prop.~7.1.5 and Thm.~5.1.3]{AmbrosioGigliSavare2008}. 
    Moreover, $\F$ is proper and lower semicontinuous because $V_i$ and $W_{ij}$ are continuous, while the entropy is lower semicontinuous by~\cite[Lem.~9.4.3]{AmbrosioGigliSavare2008}. 
    Consequently, $\F$ admits a minimizer $\mu_\star \in \Pcal(\T^d)^N$.
    Since $\F(\mu_\star) < +\infty$, the components of $\mu_\star$ are absolutely continuous.
    As~\eqref{eq:multispecies_pde} is the Wasserstein gradient flow of $\F$, $\mu_\star$ is a stationary state.
    For $\bar\mu \in \Pcal(\T^d)^N$, set $U_i[\bar\mu] \coloneqq V_i + \sum_{j=1}^N W_{ij} \ast \bar\mu_j$. 
    If $\bar\mu$ satisfies~\eqref{eq:kirkwood_monroe}, then $\sigma_i \nabla\bar\mu_i + \bar\mu_i \nabla U_i[\bar\mu] = 0$.
    Hence, $\sigma_i \Delta\bar\mu_i + \nabla \cdot (\bar\mu_i \nabla U_i[\bar\mu]) = 0$, and $\bar\mu$ is stationary.
    Conversely, assume that $\bar\mu$ is stationary. 
    Fix $i$ and let $Z_i \coloneqq \int_{\T^d} e^{-U_i[\bar\mu]/\sigma_i} \, dx$ and $v_i \coloneqq Z_i^{-1}e^{-U_i[\bar\mu]/\sigma_i}$. 
    Then $v_i \in \Pcal(\T^d) \cap C^2(\T^d)$, there exist constants $c_i, C_i > 0$ such that $c_i \le v_i \le C_i$, and $\sigma_i\Delta v_i + \nabla \cdot (v_i\nabla U_i[\bar\mu]) = 0$.
    Writing $f_i \coloneqq \bar\mu_i/v_i \in L^1(\T^d)$ gives $\nabla \cdot (v_i \nabla f_i) = 0$ in the distributional sense. 
    For every $\eta \in C^\infty(\T^d)$ with zero mean, the Lax--Milgram lemma on $\dotHk[1](\T^d)$, an adaptation of~\cite[Thm.~8.8]{gilbarg2001elliptic} to the flat torus, and elliptic bootstrap, yield $\phi \in C^2(\T^d)$, unique up to constants, satisfying $\nabla \cdot (v_i \nabla\phi) = \eta$.
    Then $0 = \int_{\T^d} f_i \nabla \cdot (v_i \nabla\phi) \, dx = \int_{\T^d} f_i\eta \, dx$, and since this holds for every $\eta \in C^\infty(\T^d)$ with zero mean, $f_i$ is constant almost everywhere. 
    As $\bar\mu_i, v_i \in \Pcal(\T^d)$, $f_i \equiv 1$, so $\bar\mu_i = v_i$ and the stated regularity and bounds follow.
\end{proof}

Fix a stationary state $\bar\mu = (\bar\mu_1, \ldots, \bar\mu_N)$ of~\eqref{eq:multispecies_pde}.
By Prop.~\ref{prop:stationary_states}, each $\bar\mu_i \in C^2(\T^d)$ is positive, so we may set $X_i \coloneqq \potential[\bar\mu_i]$ and endow $X_{\rm ms} \coloneqq \prod_{i=1}^N X_i$ with
\[
\langle \phi, \psi \rangle_{X_{\rm ms}} \coloneqq \sum_{i=1}^N \langle \phi_i, \psi_i \rangle_{\potential[\bar\mu_i]}.
\]
For $\phi, \psi \in X_{\rm ms}$, set $D_{ij} \phi(x,y) \!\coloneqq\! \nabla\phi_i(x) -\! \nabla\phi_j(y)$ and define
\begin{align} 
    &H^{\rm ms}(\phi,\psi) = \sum_{i=1}^N \int_{\T^d} \nabla^2 V_i \nabla\phi_i \cdot \nabla\psi_i \, d\bar\mu_i \label{eq:multispecies_HE} \\ 
    &+ \frac{1}{2}\sum_{i,j=1}^N \iint_{(\T^d)^2} \nabla^2 W_{ij}(x-y) D_{ij}\phi \cdot D_{ij}\psi \, d\bar\mu_i(x) d\bar\mu_j(y) \notag. 
\end{align}

\begin{proposition}
    \label{prop:multispecies_hessian_bound}
    For every stationary state $\bar\mu$, the bilinear form $H^{\rm ms}$ defined in~\eqref{eq:multispecies_HE} represents the Wasserstein Hessian of the non-entropic part of $\F$ at $\bar\mu$ and is bounded on $X_{\rm ms}$.
\end{proposition}

\begin{proof}
    The confinement and self-interaction terms follow from~\cite[Ex.~2.4]{KaliseMoschenPavliotis2026}.
    For each cross-interaction pair $(i,j)$, the same computation on the product space yields the differences $\nabla\phi_i(x) - \nabla\phi_j(y)$ and $\nabla\psi_i(x) - \nabla\psi_j(y)$.
    Summing over the ordered pairs with the pre-factor $1/2$ gives~\eqref{eq:multispecies_HE}.
    By Cauchy--Schwarz, the confinement terms are bounded by $\max_i \|\nabla^2 V_i\|_{L^\infty}\|\phi\|_{X_{\rm ms}} \|\psi\|_{X_{\rm ms}}$.
    Moreover, $\|D_{ij} \phi\|_{L^2(\bar\mu_i \otimes \bar\mu_j)} \le \sqrt{2}\|\phi\|_{X_{\rm ms}}$, and analogously for $\psi$.
    Cauchy--Schwarz therefore bounds the interaction contribution by $N^2 \max_{i,j} \|\nabla^2 W_{ij}\|_{L^\infty}\|\phi\|_{X_{\rm ms}}\|\psi\|_{X_{\rm ms}}$.
\end{proof}

Consequently, Assumptions (H1) and (H2) from~\cite{KaliseMoschenPavliotis2026}, required for the Wasserstein Hessian realization, are satisfied at every stationary state $\bar\mu$.
The analogue of Assumption~(H3) in~\cite{KaliseMoschenPavliotis2026} is also directly verified in the multi-species setting because the derivative of the first variation of the non-entropic part of $\F$ is given by $(\mathcal{K}^{\rm ms} \nu)_i = \sum_{j=1}^N W_{ij} \ast \nu_j$.

By~\cite[Prop.~3.1]{KaliseMoschenPavliotis2026} applied componentwise and Prop.~\ref{prop:multispecies_hessian_bound}, the Wasserstein Hessian of $\F$ at $\bar\mu$ is represented on $X_{\rm ms}$ by
\begin{equation} 
    \label{eq:ams_form} 
    a_{\mathrm{ms}}(\phi,\psi) \coloneqq H^{\rm ms}(\phi,\psi) + \sum_{i=1}^N \sigma_i \int_{\T^d} \nabla^2\phi_i : \nabla^2\psi_i \, d\bar\mu_i, 
\end{equation}
with form domain $D(a_{\mathrm{ms}}) = [\dotHk[2](\T^d)]^N$. 
The entropy contribution is closed and coercive, while Prop.~\ref{prop:multispecies_hessian_bound} shows that $H^{\rm ms}$ is bounded on $X_{\rm ms}$. 
Thus, the arguments of~\cite[Thms.~3.3--3.4]{KaliseMoschenPavliotis2026} extend to the finite product, and $a_{\mathrm{ms}}$ is densely defined, closed, symmetric, and lower bounded.
Furthermore, the associated self-adjoint operator $A_{\rm ms} : D(A_{\rm ms}) \subset X_{\rm ms} \to X_{\rm ms}$ has compact resolvent and is characterized by $a_{\rm ms}(\phi,\psi) = \langle A_{\rm ms}\phi,\psi\rangle_{X_{\rm ms}}$ for $\phi \in D(A_{\rm ms})$ and $\psi \in D(a_{\rm ms})$.

\section{Feedback Control Strategies}
\label{sec:hessian-feedback}

Fix a stationary state $\bar\mu$ and a target decay rate $\delta > 0$.
Following~\cite{KaliseMoschenPavliotis2026}, we seek to locally stabilize $\bar\mu$ with an exponential rate greater than $\delta$. 
Given shape control functions $\alpha_1^{\rm ms}, \ldots, \alpha_m^{\rm ms} \in X_{\rm ms}$, define $B_{\rm ms} u \coloneqq \sum_{j=1}^m u_j \alpha_j^{\rm ms}$  for $u \in \R^m$.
The controlled system reads, for $i=1,\ldots,N$,
\begin{equation}
    \label{eq:controlled_multispecies}
    \partial_t \mu_i = \nabla \cdot \left[\mu_i \nabla \left( \frac{\delta\F}{\delta\mu_i}(\mu) + (B_{\rm ms} u(t))_i \right)\right].
\end{equation}
Let $\rieszmap[\bar\mu_i] \phi \coloneqq -\nabla \cdot (\bar\mu_i \nabla\phi)$ and $\rieszmap[\bar\mu] \coloneqq \operatorname{diag}(\rieszmap[\bar\mu_1], \ldots, \rieszmap[\bar\mu_N])$.
The operator $\rieszmap[\bar\mu_i]$ is an isomorphism from $X_i$ to $X_i'$.
The linearized generator of~\eqref{eq:multispecies_pde} satisfies $L_{\rm ms} = -\rieszmap[\bar\mu] A_{\rm ms} \rieszmap[\bar\mu]^{-1}$.
Thus, setting $\nu \coloneqq \mu - \bar\mu$ and $\xi \coloneqq \rieszmap[\bar\mu]^{-1} \nu \in X_{\rm ms}$ and linearizing~\eqref{eq:controlled_multispecies} at $(\bar\mu,0)$ gives, after conjugation, $\partial_t \xi = -A_{\rm ms}\xi - B_{\rm ms}u(t)$.

\subsection{Controlling all species}

Following the strategy from~\cite{KaliseMoschenPavliotis2026}, let $(e_k^{\rm ms})_{k \ge 1}$ be an $X_{\rm ms}$-orthonormal basis of eigenfunctions of $A_{\rm ms}$, ordered so that $\lambda_1^{\rm ms} \le \lambda_2^{\rm ms} \le \cdots$ and $\lambda_k^{\rm ms} \longrightarrow +\infty$.
Define
\[
\Lambda_\delta^{\rm ms} \coloneqq \sigma(A_{\rm ms}) \cap (-\infty,\delta], \; X_\delta^{\rm ms} \coloneqq \operatorname{span}\{e_k^{\rm ms} : \lambda_k^{\rm ms} \in \Lambda_\delta^{\rm ms}\},
\]
where $\sigma(A_{\rm ms})$ is the spectrum of $A_{\rm ms}$.
Assume $\Lambda_\delta^{\rm ms}$ is not empty and choose $m = \dim X_\delta^{\rm ms}$ and take $\{\alpha_j^{\rm ms}\}_{j=1}^m$ as an orthonormal basis of $X_\delta^{\rm ms}$.
Then $B_{\rm ms} B_{\rm ms}^\ast$ is the $X_{\rm ms}$-orthogonal projection onto $X_\delta^{\rm ms}$.

Having fixed the spatial controls, we determine the time-dependent amplitudes $u_j(\cdot)$ through the product-space analogue of the infinite-horizon linear-quadratic problem in~\cite[Sec.~4.2]{KaliseMoschenPavliotis2026}.
Let $Q_{\rm ms} : X_{\rm ms} \to X_{\rm ms}$ be bounded, self-adjoint, nonnegative, diagonal in $(e_k^{\rm ms})_{k \ge 1}$, and positive on $X_\delta^{\rm ms}$. 
The associated algebraic Riccati equation is
\begin{multline}
    \label{eq:ARE-product}
    -\langle(A_{\rm ms}-\delta I)\phi,\Pi_{\rm ms}\psi\rangle_{X_{\rm ms}}
    -\langle\Pi_{\rm ms}\phi,(A_{\rm ms}-\delta I)\psi\rangle_{X_{\rm ms}}\\
    -\langle B_{\rm ms}^\ast\Pi_{\rm ms}\phi,B_{\rm ms}^\ast\Pi_{\rm ms}\psi\rangle_{\R^m}
    +\langle Q_{\rm ms}\phi,\psi\rangle_{X_{\rm ms}}=0,
\end{multline}
for all $\phi, \psi \in D(A_{\rm ms})$. 
For nonlinear stabilization, we assume the product-space analogue of Assumption~(H6) in~\cite{KaliseMoschenPavliotis2026}, which follows by adapting the maximal-regularity and remainder estimates of~\cite[Sec.~3.4]{KaliseMoschenPavliotis2025} to the finite product system.

\begin{proposition}
    \label{prop:full_modal_feedback}
    Under the above assumptions,~\eqref{eq:ARE-product} admits a unique stabilizing nonnegative self-adjoint solution $\Pi_{\rm ms}$, which is diagonal in the eigenbasis $(e_k^{\rm ms})_{k \ge 1}$.
    The feedback control $u(t) = B_{\rm ms}^\ast\Pi_{\rm ms}\xi(t)$ is optimal for the corresponding linear-quadratic problem and yields the closed-loop operator $A_{{\rm ms},\Pi} \coloneqq A_{\rm ms} + B_{\rm ms}B_{\rm ms}^\ast \Pi_{\rm ms}$, which satisfies $\lambda_\Pi \coloneqq \min\sigma(A_{{\rm ms},\Pi}) > \delta$.
    Moreover, the nonlinear closed-loop system is locally exponentially stable at $\bar\mu$ with every decay rate $\gamma \in (0, \lambda_\Pi)$.
\end{proposition}

\begin{proof}
    Since $\{\alpha_j^{\rm ms}\}_{j=1}^m$ is an orthonormal basis of $X_\delta^{\rm ms}$, $B_{\rm ms}^\ast$ is injective on $X_\delta^{\rm ms}$.
    As in~\cite[Prop.~4.3]{KaliseMoschenPavliotis2026}, this injectivity and the positivity of $Q_{\rm ms}$ on $X_\delta^{\rm ms}$ give the Hautus stabilizability and detectability conditions.
    Additionally, the argument of~\cite[Prop.~4.4]{KaliseMoschenPavliotis2026} extends to $X_{\rm ms}$, showing that $\Pi_{\rm ms}$ is diagonal in $(e_k^{\rm ms})_{k \ge 1}$ with explicitly computable coefficients.
    The conclusions follow from the product-space versions of the spectral gap and nonlinear stabilization results in~\cite[Cor.~4.5 and Thm.~4.12]{KaliseMoschenPavliotis2026}.
\end{proof}

\subsection{Multi-species stabilization via single-species control}

Having established $\delta$-stabilization when all species are actuated, we now characterize when the same rate can be achieved by controlling a single species.
Fix $r \in \{1, \ldots, N\}$ and let $P_r : X_{\rm ms} \to X_r$ and $J_r : X_r \to X_{\rm ms}$ be the coordinate projection and canonical injection, respectively, so that $J_r^\ast = P_r$.
For $\alpha = (\alpha_1, \ldots, \alpha_m) \in X_r^m$, define
\[
B_{r,\alpha}u \coloneqq \sum_{j=1}^m u_j J_r \alpha_j, \quad u \in \R^m.
\]

For $\lambda \in \R$, set $E_\lambda \coloneqq \ker(A_{\rm ms} - \lambda I)$, $d_\lambda \coloneqq \dim E_\lambda$, and $m_{r,\delta} \coloneqq \max_{\lambda \in \Lambda_\delta^{\rm ms}} d_\lambda$.
We call $\alpha$ $\delta$-stabilizing through species $r$ if the pair $(-A_{\rm ms} + \delta I, -B_{r,\alpha})$ is stabilizable.
With this notation, the Hautus criterion reads $E_\lambda \cap \ker(B_{r,\alpha}^\ast) = \{0\}$ for every $\lambda \in \Lambda_\delta^{\rm ms}$.
Since $[B_{r,\alpha}^\ast e]_j = \langle P_r e, \alpha_j \rangle_{X_r}$, this is equivalent to requiring that, for every $\lambda \in \Lambda_\delta^{\rm ms}$, the only $e \in E_\lambda$ satisfying $\langle P_r e, \alpha_j \rangle_{X_r} = 0$ for all $j$ is $e = 0$.
In particular, if $E_\lambda \cap \ker(P_r) \neq \{0\}$ for some $\lambda \in \Lambda_\delta^{\rm ms}$, then no bounded control operator $B$ with $\operatorname{Ran}(B) \subset J_r X_r$ can achieve $\delta$-stabilization.
Indeed, if $e \in E_\lambda \cap \ker(P_r)$, then $Bu = J_r \alpha^{(u)}$ for some $\alpha^{(u)} \in X_r$ and $\langle B u, e \rangle_{X_{\rm ms}} = \langle \alpha^{(u)}, P_r e \rangle_{X_r} = 0$, so $B^\ast e = 0$ and the Hautus condition fails.

\begin{theorem}
    \label{thm:delta_stabilization_one_species}
    Assume $E_\lambda \cap \ker(P_r) = \{0\}$ for every $\lambda \in \Lambda_\delta^{\rm ms}$.
    If $m \geq m_{r,\delta}$, the $\delta$-stabilizing tuples form an open dense subset of $X_r^m$.
    If $m < m_{r,\delta}$, no tuple in $X_r^m$ is $\delta$-stabilizing.
\end{theorem}

\begin{proof}
    Fix $\lambda \in \Lambda_\delta^{\rm ms}$, let $(e_{\lambda,1}, \ldots, e_{\lambda,d_\lambda})$ be an orthonormal basis of $E_\lambda$, and set $f_{\lambda,\ell} \coloneqq P_r e_{\lambda,\ell}$ for $\ell = 1, \ldots, d_\lambda$.
    For $\alpha = (\alpha_1, \ldots, \alpha_m) \in X_r^m$, define
    \[
    M_\lambda(\alpha) \coloneqq \left(\langle \alpha_j, f_{\lambda,\ell} \rangle_{X_r}\right)_{1 \le j \le m, 1 \le \ell \le d_\lambda},
    \]
    so that $B_{r,\alpha}^\ast e = M_\lambda(\alpha) c$ if $e = \sum_{\ell=1}^{d_\lambda} c_\ell e_{\lambda,\ell}$. 
    Hence, the Hautus criterion is equivalent to $\operatorname{rank} M_\lambda(\alpha) = d_\lambda$ for every $\lambda \in \Lambda_\delta^{\rm ms}$.
    Moreover, $E_\lambda \cap \ker(P_r) = \{0\}$ implies that $f_{\lambda,1}, \ldots, f_{\lambda,d_\lambda}$ are linearly independent in $X_r$.

    Suppose that $m \ge m_{r,\delta}$.
    Then $m \ge d_\lambda$ for every $\lambda \in \Lambda_\delta^{\rm ms}$. 
    For each such $\lambda$, define $p_\lambda(\alpha) \coloneqq \det(M_\lambda(\alpha)^\top M_\lambda(\alpha))$ for $\alpha \in X_r^m$.
    Since the map $\alpha \mapsto M_\lambda(\alpha)$ is continuous and linear, $p_\lambda$ is a continuous polynomial, and $\operatorname{rank} M_\lambda(\alpha) = d_\lambda$ if and only if $p_\lambda(\alpha) > 0$.
    Hence, the subset of $X_r^m$ where $\operatorname{rank} M_\lambda(\alpha) = d_\lambda$ is open.
    To establish density, first note that choosing $\alpha_j = f_{\lambda,j}$ for $1 \le j \le d_\lambda$ and $\alpha_j = 0$ for $j > d_\lambda$ gives $p_\lambda(\alpha) = \det(G_\lambda^\top G_\lambda) > 0$, where $G_\lambda$ is the Gram matrix of $f_{\lambda,1}, \ldots, f_{\lambda,d_\lambda}$. 
    Thus, $p_\lambda$ is not identically zero.
    Define $Z_\lambda \coloneqq \{\alpha \in X_r^m : p_\lambda(\alpha) = 0\}$.
    If $Z_\lambda$ contained a ball $B(\alpha^0, \varepsilon)$ with $\varepsilon > 0$, then, for any $\beta \in X_r^m$, the polynomial $q_\beta(t) \coloneqq p_\lambda(\alpha^0 + t(\beta - \alpha^0))$ would vanish for all sufficiently small $t$, hence for every $t \in \R$ because it is a polynomial in $t$.
    Thus $p_\lambda(\beta) = q_\beta(1) = 0$ for every $\beta$, contradicting $p_\lambda \not \equiv 0$ and implying that $Z_\lambda$ has an empty interior, so $X_r^m \setminus Z_\lambda$ is dense.
    Since $\Lambda_\delta^{\rm ms}$ is finite, the intersection of these open dense sets with $\lambda \in \Lambda_\delta^{\rm ms}$ shows that $\delta$-stabilizing tuples form an open dense subset of $X_r^m$.
    
    Finally, if $m < m_{r,\delta}$, choose $\lambda \in \Lambda_\delta^{\rm ms}$ attaining $d_\lambda = m_{r,\delta}$.
    For every $\alpha \in X_r^m$, $\operatorname{rank} M_\lambda(\alpha) \le m < d_\lambda$, so the Hautus condition fails. 
    Therefore, no tuple in $X_r^m$ is $\delta$-stabilizing through species $r$.
\end{proof}

The theorem motivates the definition
\[
\delta_r^\star \coloneqq \min \{\lambda \in \R : E_\lambda \cap \ker(P_r) \neq \{0\}\},
\]
with the convention $\min \emptyset \coloneqq +\infty$.
It follows that stabilization by controls acting only on species $r$ is possible for every threshold $\delta < \delta_r^\star$ and impossible for every $\delta \ge \delta_r^\star$.

\begin{corollary}
    \label{cor:single_species_riccati}
    Fix $0 < \delta < \delta_r^\star$.
    For every $m \ge m_{r, \delta}$, there exists an open dense set $\mathcal{G}_{r, \delta}^m \subset X_r^m$ such that, for every $\alpha \in \mathcal{G}_{r, \delta}^m$, the algebraic Riccati equation~\eqref{eq:ARE-product}, with $B_{\rm ms}$ replaced by $B_{r,\alpha}$, admits a unique stabilizing nonnegative self-adjoint solution $\Pi_{r,\alpha}$.
    Moreover, for every $\alpha \in \mathcal{G}_{r,\delta}^m \cap [W^{2,\infty}(\T^d)]^m$, under the hypotheses of Prop.~\ref{prop:full_modal_feedback}, the feedback $u(t) = B_{r,\alpha}^\ast \Pi_{r,\alpha} \xi(t)$ locally exponentially stabilizes the system at $\bar\mu$ with decay rate $\delta$.
\end{corollary}

\begin{proof}
    Since $\delta < \delta_r^\star$, we have $E_\lambda \cap \ker(P_r) = \{0\}$ for every $\lambda \in \Lambda_\delta^{\rm ms}$. 
    Theorem~\ref{thm:delta_stabilization_one_species} therefore shows that the $\delta$-stabilizing tuples form an open dense subset $\mathcal{G}_{r,\delta}^m$ of $X_r^m$.
    For every $\alpha \in \mathcal{G}_{r,\delta}^m$, stabilizability of $(-A_{\rm ms} + \delta I, -B_{r,\alpha})$ and positivity of $Q_{\rm ms}$ on $X_\delta^{\rm ms}$ give the stabilizability and detectability conditions required by the Riccati argument of Prop.~\ref{prop:full_modal_feedback}.
    Hence, the corresponding Riccati equation admits a unique stabilizing solution $\Pi_{r,\alpha}$ and the shifted closed-loop system is exponentially stable.
    If, in addition, $\alpha_j \in W^{2,\infty}(\T^d)$, the maximal-regularity argument of~\cite[Thm.~17]{KaliseMoschenPavliotis2025} and the feedback-remainder estimate extend to $B_{r,\alpha}$, and the nonlinear stabilization argument of Prop.~\ref{prop:full_modal_feedback} applies.
\end{proof}

\subsection{Two-species stabilization criterion}

Assume $N = 2$.
The linearized operator $L_{\rm ms}$ satisfies
\[
(L_{\rm ms}\nu)_i = -\rieszmap[\bar\mu_i] \left(\sigma_i \frac{\nu_i}{\bar\mu_i} + W_{i1} \ast \nu_1 + W_{i2} \ast \nu_2\right).
\]
Write $(L_{\rm ms}\nu)_i = \sum_{j=1}^2 L_{ij}\nu_j$ and, for $\lambda \in \R$, set $F_{2,\lambda} \coloneqq \ker(L_{22} + \lambda I)$ and $d_\lambda^{(2)} \coloneqq \dim F_{2,\lambda}$.
Since $\rieszmap[\bar\mu]$ is diagonal, $L_{22} = -\rieszmap[\bar\mu_2] A_{22} \rieszmap[\bar\mu_2]^{-1}$, where $A_{22}$ is the $(2,2)$ block of $A_{\rm ms}$. 
Hence, $\sigma(-L_{22}) = \sigma(A_{22})$ is real and discrete.

\begin{proposition}
    \label{prop:invisible_modes_two_species}
    For every $\lambda \in \R$, the spaces $E_\lambda \cap \ker(P_1)$ and $F_{2,\lambda} \cap \ker(L_{12})$ are linearly isomorphic through the map $(0, e_2) \mapsto \rieszmap[\bar\mu_2] e_2$.
\end{proposition}

\begin{proof}
    Fix $\lambda \in \R$, let $e = (0, e_2) \in E_\lambda \cap \ker(P_1)$, and set $\nu \coloneqq \rieszmap[\bar\mu] e$.
    The identity $L_{\rm ms} = -\rieszmap[\bar\mu]A_{\rm ms}\rieszmap[\bar\mu]^{-1}$ gives $L_{\rm ms}(0,\nu_2) = -\lambda(0,\nu_2)$, which is equivalent to $L_{12}\nu_2 = 0$ and $(L_{22} + \lambda I) \nu_2 = 0$.
    Conversely, if $\nu_2 \in F_{2,\lambda} \cap \ker(L_{12})$, then $L_{\rm ms}(0,\nu_2) = (L_{12}\nu_2,L_{22}\nu_2) = -\lambda(0,\nu_2)$.
    Since $\rieszmap[\bar\mu_2] : X_2 \to X_2'$ is an isomorphism, there exists a unique $e_2 \in X_2$ such that $\rieszmap[\bar\mu_2] e_2 = \nu_2$.
    The conjugacy between $L_{\rm ms}$ and $-A_{\rm ms}$ then gives $(0,e_2) \in E_\lambda \cap \ker(P_1)$.
    Hence, $(0,e_2) \mapsto \rieszmap[\bar\mu_2]e_2$ is a linear isomorphism between the two spaces.
\end{proof}

Combining Prop.~\ref{prop:invisible_modes_two_species} with Theorem~\ref{thm:delta_stabilization_one_species}, stabilization through species $1$ at rate $\delta$ is possible if and only if for every $\lambda \in \sigma(-L_{22}) \cap (-\infty, \delta]$, we have $F_{2,\lambda} \cap \ker(L_{12}) = \{0\}$.
Thus, $L_{12}$ must be injective on every eigenspace of the unactuated species whose eigenvalue does not exceed $\delta$, meaning that the cross-interaction $W_{12}$ detects every nonzero eigenfunction in these eigenspaces.
We express this condition more precisely in terms of Fourier coefficients.

For the Fourier characterization below, we complexify the spaces and operators without changing notation, and we write $\widehat{f}(k)$ for the $k$-th Fourier coefficient of $f$.
Choose a basis $(\nu_{\lambda,1}, \ldots, \nu_{\lambda,d_\lambda^{(2)}})$ of $F_{2,\lambda}$ and define, for $c \in \mathbb{C}^{d_\lambda^{(2)}}$,
\[
\mathsf{V}_{12,\lambda}c \coloneqq \left(\widehat{W}_{12}(k) \sum_{\ell=1}^{d_\lambda^{(2)}} c_\ell \widehat\nu_{\lambda,\ell}(k)\right)_{k \in \mathbb{Z}^d \setminus \{0\}}.
\]

\begin{corollary}
    \label{cor:fourier_two_species}
    The spaces $\ker(\mathsf{V}_{12,\lambda})$ and $E_\lambda \cap \ker(P_1)$ are linearly isomorphic for every $\lambda \in \sigma(-L_{22})$.
\end{corollary}

\begin{proof}
    Fix $\lambda \in \sigma(-L_{22})$.
    By Prop.~\ref{prop:invisible_modes_two_species}, the spaces $E_\lambda \cap \ker(P_1)$ and $F_{2,\lambda} \cap \ker(L_{12})$ are linearly isomorphic via the map $(0,e_2) \mapsto \rieszmap[\bar\mu_2]e_2$.
    Therefore, it remains to identify $F_{2,\lambda} \cap \ker(L_{12})$ with $\ker(\mathsf{V}_{12,\lambda})$.
    For $c \in \mathbb{C}^{d_\lambda^{(2)}}$, set $\nu_c \coloneqq \sum_{\ell=1}^{d_\lambda^{(2)}} c_\ell \nu_{\lambda,\ell}$.
    Since $(\nu_{\lambda,1}, \ldots, \nu_{\lambda,d_\lambda^{(2)}})$ is a basis of $F_{2,\lambda}$, the map $c \mapsto \nu_c$ is a linear isomorphism from $\mathbb{C}^{d_\lambda^{(2)}}$ onto $F_{2,\lambda}$.

    Now, $L_{12} \nu_c = 0$ if and only if $W_{12} \ast \nu_c$ is constant, since $-\nabla \cdot (\bar\mu_1 \nabla\cdot)$ has constant functions in its kernel.
    Since $\nu_c \in X_2'$ has zero mean, so does $W_{12} \ast \nu_c$, and thus $L_{12}\nu_c = 0$ if and only if $W_{12} \ast \nu_c = 0$, which is equivalent to $[\mathsf{V}_{12,\lambda}c]_k = 0$ for every $k \in \mathbb{Z}^d \setminus \{0\}$.
    Therefore, $c \mapsto \nu_c$ restricts to a linear isomorphism between $\ker(\mathsf{V}_{12,\lambda})$ and $F_{2,\lambda} \cap \ker(L_{12})$.
    Combining this with Prop.~\ref{prop:invisible_modes_two_species} proves the result.
\end{proof}

As a consequence of the previous corollary, 
\[
\dim \left(E_\lambda \cap \ker(P_1)\right) = \dim \ker(\mathsf{V}_{12,\lambda}),
\]
and hence $E_\lambda \cap \ker(P_1) = \{0\}$ if and only if $\operatorname{rank} \mathsf{V}_{12,\lambda} = d_\lambda^{(2)}$.
Thus, for $\lambda \le \delta$, the condition from Theorem~\ref{thm:delta_stabilization_one_species} can be verified directly from the Fourier coefficients of $W_{12}$ and the eigenfunctions of $L_{22}$.

\section{Examples and Numerical Experiments}
\label{sec:numerics}

Given the periodic geometry of the torus, we compute the eigenfunctions of $A_{\rm ms}$ through a Fourier--Galerkin discretization on $\T$ with basis $\{\cos(2\pi k \cdot), \sin(2\pi k \cdot)\}_{k=1}^K$.
The nonlinear dynamics is then projected onto the first $40$ computed eigenfunctions of $A_{\rm ms}$, with the Galerkin matrices and nonlinear terms evaluated by trapezoidal quadrature on uniform grids.
In both experiments, we take $\delta = 5$ and $Q_{\rm ms} = \delta^2 I$.
The implementation adapts~\cite{KaliseMoschenPavliotis2025} to the product-space setting; further details and code are available in the GitHub repository \url{https://github.com/lucasmoschen/feedback-control-gradient-flows}.

\subsection{The two-community noisy Kuramoto model}

Motivated by the two-community noisy Kuramoto model studied in~\cite{Meylahn2020}, we consider two populations with no confinement $V_i = 0$, interaction kernels $W_{ij}(x) = -K_{ij}\cos(2\pi x)$, where $K_{11} = K_{22} = K_s$ and $K_{12} = K_{21} = K_c$, and unit diffusion, $\sigma_i = 1$.

\begin{corollary}
    \label{cor:cosine_visibility_dichotomy}
    Assume $K_c \neq 0$.
    At the stationary state $\bar\mu_1 = \bar\mu_2 = 1$, $\delta_1^\star = 16\pi^2$, whereas $\delta_1^\star = +\infty$ at every stationary state with nonuniform $\bar\mu_2$.
\end{corollary}

\begin{proof}
    The eigenvalues of $A_{\rm ms}$ with an eigenfunction $e \in \ker(P_1)$ necessarily belong to $\sigma(-L_{22})$.
    By Cor.~\ref{cor:fourier_two_species}, they are precisely those $\lambda \in \sigma(-L_{22})$ for which $\mathsf{V}_{12,\lambda}$ is not injective.
    Hence, $\delta_1^\star$ is the minimum of such eigenvalues.
    
    Assume first that $\bar\mu_1 = \bar\mu_2 = 1$.
    For $k \in \mathbb{Z} \setminus \{0\}$, set $\zeta_k(x) \coloneqq e^{2\pi \mathrm{i} k x}$.
    Since $L_{22}y = \Delta y + 4 \pi^2 K_s \cos(2\pi\cdot) \ast y$, we obtain $-L_{22}\zeta_k = a_k \zeta_k$ with $a_k \coloneqq 4 \pi^2 k^2 - 2 \pi^2 K_s \mathbf{1}_{\{|k|=1\}}$.
    Because $(\zeta_k)_{k \neq 0}$ is a basis of the complexified space $X_2'$, $\sigma(-L_{22}) = \{a_k : k \neq 0\}$ and $F_{2,\lambda} = \operatorname{span}\{\zeta_k : a_k = \lambda\}$, with basis $(\zeta_\ell)_{a_\ell = \lambda}$.
    Thus,
    \[
    \mathsf{V}_{12,\lambda}c = -\frac{K_c}{2} \left(\mathbf{1}_{\{|k|=1\}} \sum_{a_\ell = \lambda} c_\ell \ind_{\{k=\ell\}} \right)_{k \in \mathbb{Z} \setminus \{0\}}.
    \]
    Fix $\lambda \in \sigma(-L_{22})$ and suppose first that $\lambda \neq 4 \pi^2 n^2$ for all $n \ge 2$. 
    Then $a_\ell = \lambda$ if and only if $|\ell| = 1$, and therefore $F_{2,\lambda} = \operatorname{span}\{\zeta_1, \zeta_{-1}\}$.
    Hence, $\mathsf{V}_{12,\lambda}c = 0$ implies $c_1 = c_{-1} = 0$, and $\mathsf{V}_{12,\lambda}$ is injective.
    Suppose now that $\lambda = 4 \pi^2 n^2$ for some $n \ge 2$.
    Since $a_n = \lambda$, we have $\zeta_{\pm n} \in F_{2,\lambda}$ and hence $\mathsf{V}_{12,\lambda}$ is not injective.
    In conclusion, $\mathsf{V}_{12,\lambda}$ only fails to be injective for $\lambda \in \{4 \pi^2 n^2 : n \ge 2\}$, yielding $\delta_1^\star = 16\pi^2$.

    Now suppose that $\bar\mu_2$ is nonuniform.
    Fix $\lambda \in \sigma(-L_{22})$, take $c \in \ker(\mathsf{V}_{12,\lambda})$, and write $\nu \coloneqq \sum_{\ell=1}^{d_\lambda^{(2)}} c_\ell \nu_{\lambda,\ell}$.
    Since $\widehat{W}_{12}(k) \neq 0$ if and only if $|k| = 1$, the inclusion $c \in \ker(\mathsf{V}_{12,\lambda})$ implies $\widehat\nu(\pm 1) = 0$.
    It follows that $W_{22} \ast \nu = 0$.
    Moreover, $\widehat\nu(0) = 0$ because $\nu \in X_2'$.
    Now, set $z \coloneqq (\log\bar\mu_2)' = -W_{12}' \ast \bar\mu_1 - W_{22}' \ast \bar\mu_2$, so that
    \[
    \widehat{z}(k) = -2\pi \widehat{\sin(2\pi \cdot)}(k) \left(K_c \widehat{\bar\mu}_1(k) + K_s \widehat{\bar\mu}_2(k)\right),
    \]
    and hence $\widehat{z}(k) = 0$ for all $k \neq \pm 1$.
    Since $\bar\mu_2$ is nonuniform, $z \not \equiv 0$, and because $z$ is real-valued, $\widehat{z}(-1) = \overline{\widehat{z}(1)}$. 
    Therefore $\widehat{z}(\pm 1) \neq 0$.
    Since $L_{22}\nu = -\lambda\nu$ and $W_{22} \ast \nu = 0$, we obtain $-\nu'' + (z\nu)' = \lambda\nu$.
    Taking Fourier coefficients gives
    \[
    (4\pi^2 k^2 - \lambda) \widehat\nu(k) + 2\pi\mathrm{i} k (\widehat{z}(1) \widehat\nu(k-1) + \widehat{z}(-1) \widehat\nu(k+1)) = 0
    \]
    for all $k \in \mathbb{Z}$.
    We know that $\widehat\nu(0) = \widehat\nu(\pm 1) = 0$.
    Assume inductively that $\widehat\nu(k) = 0$ for every $k \in \{-n, \ldots, n\}$, where $n \ge 1$.
    Evaluating the identity at $k = n$ and $k = -n$ yields $2\pi \mathrm{i} n \widehat{z}(-1) \widehat\nu(n+1) = 0$ and $-2\pi \mathrm{i} n \widehat{z}(1) \widehat\nu(-n-1) = 0$. 
    Since $\widehat{z}(\pm 1) \neq 0$, it follows that $\widehat\nu(n+1) = \widehat\nu(-n-1) = 0$.
    By induction, $\widehat\nu(k) = 0$ for all $k \in \mathbb{Z}$.
    Thus, $\nu = 0$, and the linear independence of $\nu_{\lambda, 1}, \ldots, \nu_{\lambda, d_\lambda^{(2)}}$ gives $c = 0$. 
    Therefore, $\mathsf{V}_{12,\lambda}$ is injective for every $\lambda \in \sigma(-L_{22})$, and hence $\delta_1^\star = +\infty$.
\end{proof}

For this experiment, the target equilibrium is the uniform distribution $\bar\mu_i \equiv 1$, which is unstable for the chosen value $K_s = 3$; see~\cite[Thm.~1.3]{CarrilloSalmaniw2025}.
For the positive couplings considered here, the attractive cosine interactions drive the uncontrolled dynamics toward synchronization.
For $K_c = 2$, both feedback strategies raise the lower spectral bound from approximately $-60$ to $20$.
Accordingly, the controlled nonlinear trajectories decay toward the homogeneous equilibrium.
We quantify synchronization using the order parameters
\[
r_j(t) = \left|\int_\T e^{2\pi \mathrm{i} x} \mu_j(t,x) \, dx\right|, \quad j = 1,2.
\]
As shown in Fig.~\ref{fig:kuramoto_control}, both feedback strategies suppress synchronization for $K_c = 2$ at essentially the same rate.
Since $\delta = 5 < \delta_1^\star = 16\pi^2$, Cor.~\ref{cor:cosine_visibility_dichotomy} guarantees stabilizability whenever $K_c \neq 0$. 
Panel~(c) highlights the role of the cross-interaction: when $K_c = 0$, species $2$ synchronizes, whereas for $K_c = 1, 2$ its order parameter decays, with faster decay under stronger coupling.

\begin{figure}[t]
    \centering
    \includegraphics[width=\linewidth]{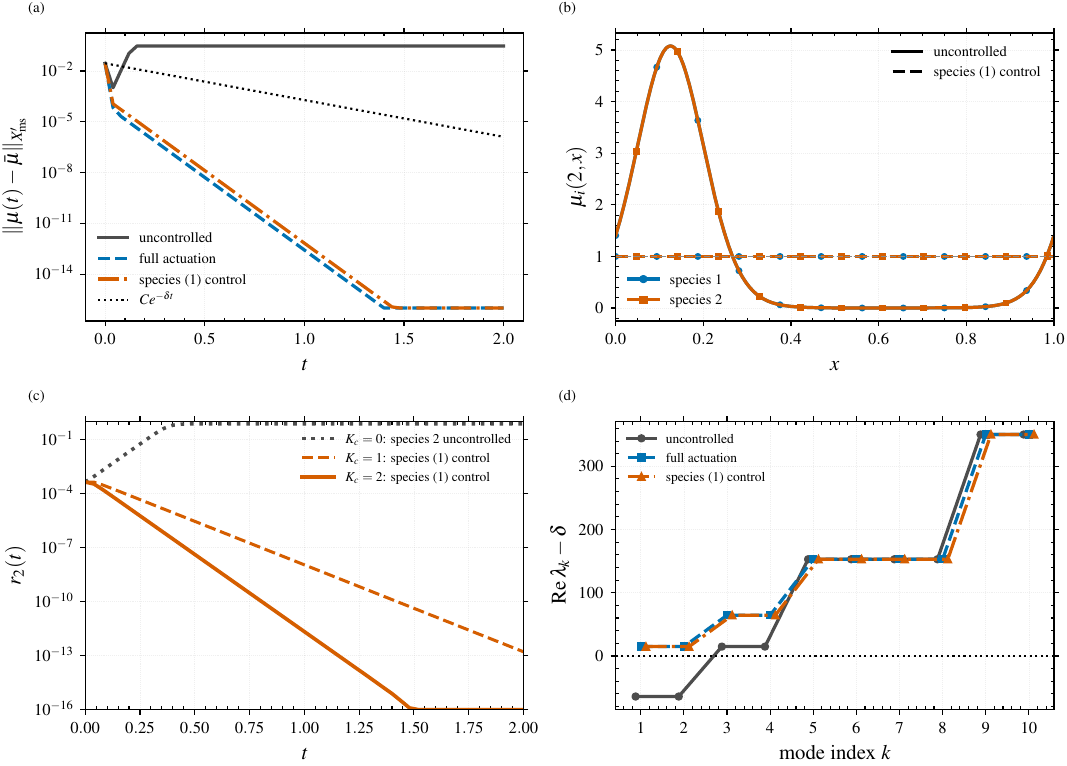}
    \caption{\textbf{Stabilization of the two-population Kuramoto system.} 
    The parameters are $K_s = 3$, $\delta = 5$, and initial densities having two peaks, centered at $0.25$ and $0.75$ for species $1$, and at $0$ and $0.5$ for species $2$.
    Panels (a), (b), and (d) use $K_c = 2$, whereas panel (c) considers $K_c \in \{0,1,2\}$.
    Panel (a): distance to equilibrium in $X_{\rm ms}'$-norm.
    Panel (b): final densities. 
    Panel (c): order parameter $r_2(t)$ of the unactuated species under single-species control. 
    Panel (d): first ten shifted eigenvalues.}
    \label{fig:kuramoto_control}
\end{figure}

\subsection{Cell-sorting of two species}

Our second experiment is based on a movement model for the spatial organization of interacting species, studied in~\cite{GiuntaHillenLewisPotts2024}, in which each species has nonlocal information about the spatial distribution of the other species.
In our notation, we consider the rescaled case $V_i = 0$, $\sigma_i = 1$, $W_{11} = W_{22} = 0$, and $W_{12} = W_{21} = \gamma K$, where $K(y) = \frac{1}{2}\mathbf{1}_{\{|y|_{\T} \le 1/4\}}$.
Here, $\gamma \in \R$ is the cross-interaction coefficient, with $\gamma > 0$ corresponding to mutual avoidance and $\gamma < 0$ to mutual attraction.
To meet our regularity assumptions, we replace $K$ by the periodic heat regularization $K_\tau = e^{2\pi^2 \tau^2}G_\tau \ast K$, where $\widehat{G}_\tau(k) = e^{-2 \pi^2 \tau^2 k^2}$. 

The homogeneous state $\bar\mu = (1,1)$ is locally asymptotically stable for $|\gamma| < \gamma_c$ and unstable for $|\gamma| > \gamma_c$~\cite[Thms.~2.9--2.10]{CarrilloSalmaniw2025}, where $\gamma_c = 2\pi$ is the critical magnitude.
At $\gamma = \gamma_c$, a supercritical branch of nonuniform equilibria bifurcates into the region $\gamma > \gamma_c$.
For $\gamma > \gamma_c$ sufficiently close to $\gamma_c$, the equilibria on this branch have components out of phase and are locally asymptotically stable in the invariant subspace of even densities.
Such profiles represent spatial sorting, also called segregation, since regions where one species dominates have less of the other~\cite{CarrilloSalmaniw2025}.

\begin{corollary}
    \label{cor:cell_sorting_visibility_dichotomy}
    Assume $\gamma \neq 0$. 
    Then $\delta_1^\star = 16\pi^2$ at the stationary state $\bar\mu_1 = \bar\mu_2 = 1$, whereas $\delta_1^\star = +\infty$ at every stationary state with nonuniform $\bar\mu_2$.
\end{corollary}

\begin{proof}
    We compute, for $k \in \mathbb{Z} \setminus \{0\}$,
    \[
    \widehat{K}_\tau(k) = \frac{e^{-2\pi^2\tau^2(k^2-1)}}{2\pi k} \sin\left(\frac{\pi k}{2}\right),
    \]
    so $\widehat{W}_{12}(k) \neq 0$ if and only if $k$ is odd.
    
    First assume that $\bar\mu_1 = \bar\mu_2 = 1$.
    The Fourier basis functions $\zeta_k(x) \coloneqq e^{2\pi\mathrm{i}k x}$ satisfy $-L_{22} \zeta_k = 4 \pi^2 k^2 \zeta_k$.
    Therefore, $\sigma(-L_{22}) = \{4 \pi^2 n^2\}_{n \ge 1}$, $F_{2, 4\pi^2n^2} = \operatorname{span}\{\zeta_{-n}, \zeta_n\}$,
    and 
    \[
    \mathsf{V}_{12, 4\pi^2 n^2} c = \left(\widehat{W}_{12}(k)(c_1 \ind_{\{k=-n\}} + c_2 \ind_{\{k=n\}})\right)_{k \in \mathbb{Z} \setminus \{0\}}.
    \]
    Hence, $\mathsf{V}_{12,4\pi^2n^2}$ is injective if and only if $\widehat{W}_{12}(\pm n) \neq 0$, that is, if and only if $n$ is odd.
    As in the proof of Cor.~\ref{cor:cosine_visibility_dichotomy}, $\delta_1^\star$ is the smallest eigenvalue at which injectivity fails, which is the one with $n = 2$, hence $\delta_1^\star = 16\pi^2$.
    
    Suppose now that $\bar\mu_2$ is nonuniform.
    Fix $\lambda \in \sigma(-L_{22})$, take $c \in \ker(\mathsf{V}_{12,\lambda})$, and set $\nu \coloneqq \sum_{j=1}^{d_\lambda^{(2)}} c_j \nu_{\lambda,j}$.
    Since $\widehat{W}_{12}(k) \neq 0$ for every odd integer $k$, the definition of $\mathsf{V}_{12,\lambda}$ gives $\widehat\nu(k) = 0$ for every such $k$.
    Moreover, $\widehat\nu(0) = 0$ because $\nu \in X_2'$.
    The identity $L_{22}\nu = -\lambda\nu$ becomes $-\nu'' + (z\nu)' = \lambda\nu$, where $z \coloneqq (\log\bar\mu_2)' = -\gamma K_\tau' \ast \bar\mu_1$ satisfies $\widehat{z}(k) = 0$ for every even integer $k$.
    Additionally, $\nu''$ satisfies $\widehat{\nu''}(k) = -4 \pi^2 k^2 \widehat{\nu}(k) = 0$ for every odd integer $k$, whereas $\widehat{(z \nu)'}(k) = 0$ if $k$ is an even integer.
    Separating the even and odd Fourier modes gives $-\nu'' = \lambda \nu$ and $(z\nu)' = 0$.
    The second identity implies that $z \nu$ is constant with $\widehat{z\nu}(0) = 0$, which implies $z \nu = 0$.
    The nonuniformity of $\bar\mu_2$ implies $z \not \equiv 0$.
    By continuity, $z$ does not vanish on some open interval, and hence $z\nu = 0$ implies that $\nu = 0$ there.
    Since $-\nu'' = \lambda\nu$, it follows that $\nu = 0$ identically.
    Consequently, $c = 0$, and $\mathsf{V}_{12,\lambda}$ is injective for every $\lambda \in \sigma(-L_{22})$. 
    Hence, $\delta_1^\star = +\infty$.
\end{proof}

As discussed above, for every $\gamma > \gamma_c$ sufficiently close to $\gamma_c$, there exists a nonuniform equilibrium $\bar\mu$ whose two components are out of phase.
We select the even representative among the translations of $\bar\mu$.
Although this equilibrium is locally asymptotically stable in this class, its spectral gap is small near $\gamma_c$, and convergence of the uncontrolled dynamics can therefore be slow.
We therefore use feedback control to accelerate convergence to $\bar\mu$.
Since $\bar\mu_2$ is nonuniform, Cor.~\ref{cor:cell_sorting_visibility_dichotomy} gives $\delta_1^\star = +\infty$. 
Hence, for every $\delta > 0$, exponential stabilization is achievable by acting only on species $1$, provided that $m \ge m_{1,\delta}$ and the control profiles are chosen from the corresponding open dense set.

For the numerical experiments, we fix $\tau = 0.04$ and $\gamma = (1 + \varepsilon) \gamma_c$, where $\varepsilon \in \{0.005, 0.01\}$.
For each value of $\varepsilon$, we compute $\bar\mu$ using a relaxed fixed-point algorithm.
After quotienting out the translational symmetry, the spectral gap is given by the first nonzero eigenvalue, yielding $0.3952$ for $\varepsilon = 0.005$ and $0.7914$ for $\varepsilon = 0.01$.
Spectral quantities are reported for the formulation in $\T$, whereas the figures and the times below use the original variables of~\cite{GiuntaHillenLewisPotts2024}, in which lengths are multiplied by $4$ and times by $16$.
To monitor pattern formation, let $c(t) = \mu_1(t) - \mu_2(t)$ and $\bar{c} = \bar\mu_1 - \bar\mu_2$, and define $S(t) \coloneqq \langle c(t), \bar{c} \rangle_{L^2} / \|\bar{c}\|_{L^2}^2$.
Thus, we have $S = 0$ at the uniform density and $S = 1$ at the selected stationary state. 
At $t = 300$, the uncontrolled coefficient is only $0.375$ for $\varepsilon = 0.005$, compared with $0.998$ for $\varepsilon = 0.01$.
By contrast, as shown in Fig.~\ref{fig:cell_sorting_control}, both control strategies reach $S(t) \ge 0.99$ very quickly, after which the control becomes negligible.

\begin{figure}[t]
    \centering
    \includegraphics[width=\linewidth]{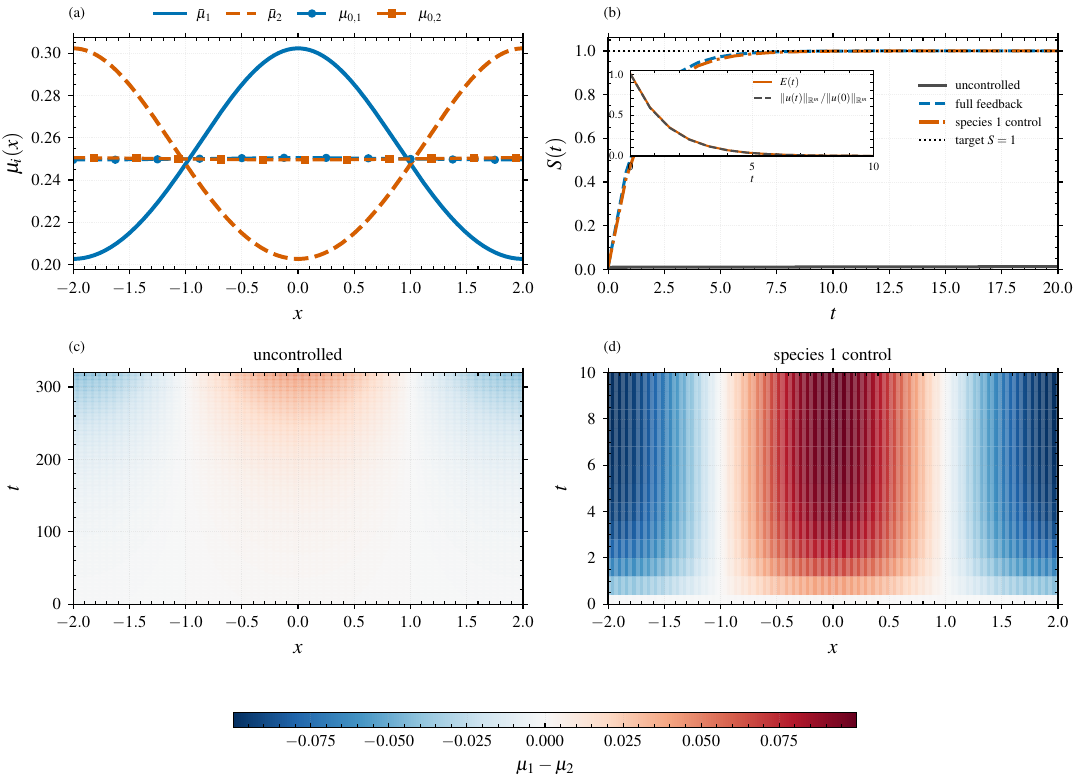}
    \caption{\textbf{Feedback control accelerates cell-sorting dynamics.}
    Results for $\varepsilon = 0.005$ and initial density $\mu_0 = (1 - \eta) \bar\mu + \eta(1, 1)$ with $\eta = 0.99$.
    The spatial panels are shown on $[-2,2]$, the scaling of~\cite{GiuntaHillenLewisPotts2024}.
    Panel (a): $\mu_0$ and $\bar\mu$.
    Panel (b): coefficient $S(t)$.
    The inset shows $E(t) = \|\mu(t) - \bar\mu\|_{X_{\rm ms}'} / \|\mu_0 - \bar\mu\|_{X_{\rm ms}'}$ and the normalized control norm $\|u(t)\|_{\R^m}/\|u(0)\|_{\R^m}$.
    Panels (c) and (d): space-time contrast $\mu_1 - \mu_2$.}
    \label{fig:cell_sorting_control}
\end{figure}

\section{Conclusion}
\label{sec:conclusion}

We extended the stabilizing feedback construction based on the Wasserstein Hessian developed in \cite{KaliseMoschenPavliotis2026} to multi-species gradient flows, obtaining local stabilization of stationary states at a prescribed exponential rate.
The analysis shows that controlling a single species can locally stabilize the dynamics when the interactions transmit the control effect to all relevant modes.
Moreover, we established that the shape control functions can be chosen from an open dense set, suggesting a trade-off between computing Hessian eigenfunctions to design structured controls and using more generic profiles at the cost of solving a less structured Riccati equation.
This trade-off deserves further investigation, as does applying the proposed methodology to biologically inspired problems, where external stimulation can suppress or promote synchronization.

\bibliographystyle{IEEEtran}
\bibliography{biblio}

\end{document}